\documentclass[11pt]{amsart}

  \usepackage[utf8]{inputenc}

  \usepackage{amsmath}
  \usepackage{amsfonts}
  \usepackage{amssymb}
  \usepackage{amsthm}
  \usepackage{scalerel}
\usepackage{stackengine,wasysym}

  \usepackage{mathrsfs}
  \usepackage{tikz}

  \usepackage[a4paper, bindingoffset=1cm]{geometry}

  \usetikzlibrary{matrix,arrows,calc}
  
\usepackage{xfrac}
\usepackage{amsmath}
\usepackage{amstext}
\usepackage{amssymb}
\usepackage{amsthm}
\usepackage{mathtools}
\usepackage{mathrsfs}
\usepackage{microtype}
\usepackage{bbm}
\newtheorem{theo}{Theorem}
\newtheorem{lemma}[theo]{Lemma}

\newtheorem{prop}[theo]{Proposition}
\newtheorem{cor}[theo]{Corollary} 

\newtheorem{defi}[theo]{Definition}

\newtheorem*{theo*}{Theorem}

\theoremstyle{remark}
\newtheorem*{rem}{Remark} 

\newcommand{\matN}{\ensuremath {\mathbb{N}}}
\newcommand{\matQ}{\ensuremath {\mathbb{Q}}}
\newcommand{\matR} {\ensuremath {\mathbb{R}}}
\newcommand{\matC} {\ensuremath {\mathbb{C}}}

\newcommand{\matZ} {\ensuremath {\mathbb{Z}}}

\renewcommand{\epsilon}{\ensuremath \varepsilon}
\renewcommand{\bar}[1]{\ensuremath \overline{#1}}
\newcommand{\res}{\operatorname{res}}
\newcommand{\Quot}{\operatorname{Quot}}
\newcommand{\im}{\operatorname{im}}

\begin{document}
 \title[]{
\textbf{THE MAILLOT-R\"OSSLER CURRENT AND THE POLYLOGARITHM ON ABELIAN SCHEMES}
}
\author{GUIDO KINGS AND DANNY SCARPONI} 
\address{Fakult\"at f\"ur Mathematik \\
Universit\"at Regensburg\\
93040 Regensburg\\
Germany}
\thanks{This research was supported by the DFG grant: SFB 1085 “Higher invariants”}

\maketitle

\begin{abstract}
 We give  a conceptual proof of the fact that the realisation of 
 the degree zero part of the polylogarithm on abelian schemes in analytic Deligne cohomology 
 can be described in terms of the Bismut-Köhler higher analytic
 torsion form of the Poincar\'e bundle. Furthermore, we provide a new axiomatic characterization of  the arithmetic Chern character of the Poincaré bundle using only invariance properties under isogenies. For this we obtain a decomposition result for the arithmetic Chow group of independent interest.
\end{abstract}
\tableofcontents

\section*{Introduction}
In an important contribution Maillot and R\"ossler constructed a Green current $\mathfrak{g}_{\mathcal{A}^\vee}$ for the zero section of an abelian scheme $\mathcal{A}$ which is norm compatible (i.e. $[n]_*\mathfrak{g}_{\mathcal{A}^\vee}=\mathfrak{g}_{\mathcal{A}^\vee}$) and is the push-forward of the arithmetic Chern character of the (canonically metrized) Poincar\'e bundle. In particular, on the complement of the zero section the Green current $\mathfrak{g}_{\mathcal{A}^\vee}$ is the degree $(g-1)$ part of the analytic torsion form of the Poincar\'e bundle. Moreover, certain linear combinations of translates of these currents are even motivic in the sense that their classes in analytic Deligne cohomology are in the image of the regulator from motivic cohomology. 

In the special case of a family of elliptic curves the current 
$\mathfrak{g}_{\mathcal{A}^\vee}$ is described by a Siegel-function whose usefulness for many arithmetic problems (in particular for special values of $L$-functions and Iwasawa theory) is well-known and one could hope that the Maillot-R\"ossler current plays a similar role for abelian schemes.

On the other hand the first author has constructed the motivic polylogarithm $\textnormal{pol}^0\in H^{2g-1}_{\mathcal{M}}(\mathcal{A}\setminus\mathcal{A}[N],g)$
of the abelian scheme $\mathcal{A}$ without its $N$-torsion points $\mathcal{A}\setminus\mathcal{A}[N]$, see \cite{kingsross}. Here $g$ is the relative dimension of $\mathcal{A}$. The polylogarithm is also norm-compatible $[n]_*\textnormal{pol}^0=\textnormal{pol}^0$ for $n$ coprime to $N$ and in the elliptic case it is directly related to Siegel functions and modular units. 

It is  natural to ask how $\textnormal{pol}^0$ is related to $\mathfrak{g}_{\mathcal{A}^\vee}$. 
This question was answered completely in \cite{kingsross} and it turns out that the image of $ -2\textnormal{pol}^0$ in analytic Deligne cohomology is the Maillot-R\"ossler current $[N]^{*}\mathfrak{g}_{\mathcal{A}^\vee}-N^{2g}\mathfrak{g}_{\mathcal{A}^\vee}$. Due to the fact that in analytic Deligne cohomology there is no residue sequence, the proof of this fact in \cite{kingsross} was much more complicated than it should be and proceeded by a reduction to the case of a product of elliptic curves via the moduli space of abelian varieties and an explicit computation. 

In this paper we give a much simpler and very structural proof of the identity between the polylogarithm and the Maillot-R\"ossler current (cf. Theorem \ref{theo}).
We circumvent the difficulties of the approach in \cite{kingsross}  by working in Betti cohomology instead of analytic Deligne cohomology. As a result one has only to compare the residues of the classes.
In fact we achieve much more and give an axiomatic characterization of the Maillot-R\"ossler current which does not involve the Poincar\'e bundle (cf. Theorem \ref{uniq}). More precisely, we prove that any class $\widehat{\xi}\in \widehat{\textnormal{CH}}^g (\mathcal{A})_\matQ $ in the arithmetic Chow group which satisfies that its image in the Chow group ${\textnormal{CH}}^g (\mathcal{A})_\matQ $ is the zero section and such that 
\[ 
([n]^*-n^{2g})(\widehat{\xi})=0 \mbox{ in }\widehat{\textnormal{CH}}^g (\mathcal{A})_\matQ 
\]
holds for some   $n\geq 2$ is in fact equal to 
$(-1)^g p_{1,\ast}
  \left(\widehat{\textnormal{ch}}(\bar{\mathcal{P}})\right)^{(g)}$. This characterization of the Maillot-R\"ossler current relies on a decomposition into generalized eigenspaces for the action of $[n]^{*}$ on the arithmetic $\widehat{\textnormal{CH}}^g (\mathcal{A})_\matQ$, which might be of independent interest (cf. Corollary \ref{cor:exact-eigenspaces}).

Here is a short synopsis of our paper: in the first section we give
some background on motivic cohomology and arithmetic Chow groups. In the second section we review 
the polylogarithm and the Maillot-R\"ossler current. In the third section we carry out the comparison between the Maillot-R\"ossler current and the polylogarithm. In the fourth section we prove a decomposition of the arithmetic Chow group and in the last section we give an axiomatic characterization of the Maillot-R\"ossler current.

The authors would like to express their gratitude to Jos\'e Burgos, who was always ready for  encouraging discussions and who explained to us many subtler aspects of his work on arithmetic Chow groups and Deligne cohomology. This work was carried out at the SFB 1085  “Higher invariants” in Regensburg whose support is gratefully acknowledged.

 \section{Preliminaries on motivic cohomology, Arakelov theory and Deligne cohomology}
\subsection{Motivic cohomology}
Let $\pi :\mathcal{A}\rightarrow S$ be an abelian scheme of relative dimension $g$, let $\epsilon: S \rightarrow\mathcal{A}$ be the zero section, let $N>1$ be an integer and let 
$\mathcal{A}[N]$ be the finite group scheme of $N$-torsion points. Here $S$ is smooth over a subfield $k$ of the complex numbers. 
We will write $S_0$ for the image of $\epsilon$ 
in $\mathcal{A}$. We denote by $\mathcal{A}^\vee $ the dual abelian scheme of $\mathcal{A} $ and by $\epsilon^\vee $ its zero section. 

In \cite{soul85} and \cite{beil}, C. Soul\'e and A. Beilinson defined motivic cohomology for any variety $V$ over a field
\[
 H^{i}_{\mathcal{M}}(V,j):=\textnormal{Gr}^j_{\gamma}K_{2j-i}(V)\otimes \matQ.
\]
\begin{rem} In this paper we work with the above rather old fashioned definition of motivic cohomology for the compatibility with earlier references. This and the requirement that $S$ is smooth over a base field $k$ is not necessary. The later condition does no harm as we are mainly interested in the arithmetic Chow groups and Deligne cohomology.
For a much more general setting we refer to the paper \cite{Huber-Kings} where also
the decomposition of motivic cohomology is considered in an up to date fashion. 
\end{rem}

For any integer $a>1$ and any $W\subseteq \mathcal{A}$ open sub-scheme such that 
\[
 j:[a]^{-1}(W)\hookrightarrow W
\]
is an open immersion (here $[a]:\mathcal{A}\rightarrow \mathcal{A}$ is the $a$-multiplication on $\mathcal{A}$), the trace map with respect to $a$ is defined as
\begin{equation}\label{eq:tr-def}
\begin{tikzpicture}[description/.style={fill=white,inner sep=2pt}, ciao/.style={fill=red,inner sep=2pt}]
\matrix (m) [matrix of math nodes, row sep=3.5em,
column sep=2.5em, text height=1.5ex, text depth=0.25ex]
{ \textnormal{tr}_{[a]}:H^{\cdot}_{\mathcal{M}}(W,*) & H^{\cdot}_{\mathcal{M}}([a]^{-1}(W),*)
& H^{\cdot}_{\mathcal{M}}(W,*). \\
  };
	\path[->,font=\scriptsize] 
		(m-1-1) edge node[above] {$ j^* $} (m-1-2)
		(m-1-2) edge node[above] {$ [a]_* $} (m-1-3)
		;
\end{tikzpicture}
\end{equation}

For any integer $r$ we let
\[
 H^{\cdot}_{\mathcal{M}}(W,*)^{(r)}:=\left\{ \psi\in H^{\cdot}_{\mathcal{M}}(W,*)|( \textnormal{tr}_{[a]}-a^r\textnormal{Id})^k\psi=0 \textnormal{ for some } k\geq 1 \right\}
\]
be the generalized eigenspace of $ \textnormal{tr}_{[a]}$ of weight $r$. One can prove that there is a decomposition into $ \textnormal{tr}_{[a]}$-eigenspaces
\[
 H^{\cdot}_{\mathcal{M}}(\mathcal{A},*)\cong \bigoplus_{r=0}^{2g}H^{\cdot}_{\mathcal{M}}(\mathcal{A},*)^{(r)}
\]
which is independent of $a$ and that 
\[
 H^{\cdot}_{\mathcal{M}}(\mathcal{A}\setminus S_0,*)^{(0)}=0
\]
 (cf. Proposition 2.2.1 in \cite{kingsross}).

\subsection{Arithmetic varieties}
  An arithmetic ring is a triple  $(R,\Sigma,F_\infty)$ where 
\begin{itemize}
 \item $R$ is an excellent regular Noetherian integral domain,
 \item $\Sigma$ is a finite nonempty set of monomorphisms $\sigma:R\rightarrow \matC$,
 \item $F_\infty$ is an anti-linear involution of the $\matC-$algebra $\matC^{\Sigma}:=\matC \underbrace{\times ... \times}_{|\Sigma|}\matC$, such that the diagram
\begin{center}
\begin{tikzpicture}[description/.style={fill=white,inner sep=2pt}, ciao/.style={fill=red,inner sep=2pt}]
\matrix (m) [matrix of math nodes, row sep=3.5em,
column sep=3.5em, text height=1.5ex, text depth=0.25ex]
{ R & \matC^{\Sigma}  \\
  R & \matC^{\Sigma}  \\
  };
	\path[->,font=\scriptsize]
		(m-1-1) edge node[left] {$ \textnormal{Id} $} (m-2-1)
		
		(m-1-2) edge node[auto] {$ F_{\infty} $} (m-2-2)
		
		;
	\path[->,font=\scriptsize] 
		(m-1-1) edge node[above] {$ \delta $} (m-1-2)
		(m-2-1) edge node[above] {$ \delta $} (m-2-2)
		
		;
\end{tikzpicture}
\end{center}
commutes (here by $\delta$ we mean the natural map to the product induced by the family of maps $\Sigma$). 
\end{itemize}
An arithmetic variety $X$ over $R$, is a scheme of finite type over $R$, which is flat, quasi-projective and regular.
As usual we write
\[
 X(\matC):=\coprod_{\sigma\in \Sigma}(X\times_{R,\sigma}\matC)(\matC).
\]
Note that $F_\infty$ induces an involution $F_\infty:X(\matC)\rightarrow X(\matC)$.
\subsection{Arithmetic Chow groups} Let $p\in \matN$.
We denote by:
\begin{itemize}
 \item $E^{p,p}(X_\matR)$ the $\matR$-vector space of 
smooth real forms $\omega$ on $X(\matC)$ of type $(p,p)$
such that $F^*_\infty \zeta=(-1)^{p}\omega$,
 \item  $\tilde E^{p,p}(X_\matR)$ the quotient $E^{p,p}(X_\matR)/(\textnormal{Im}\partial 
+\textnormal{Im}\bar\partial)$,
 \item $D^{p,p}(X_\matR)$ the $\matR$-vector space of real currents 
$\zeta$ on $X(\matC)$ of type $(p,p)$ such that $F^*_\infty \zeta=(-1)^{p}\zeta$,
 \item  $\tilde D^{p,p}(X_\matR)$ the quotient $D^{p,p}(X_\matR)/(\textnormal{Im}\partial 
+\textnormal{Im}\bar\partial)$.
 \end{itemize}
 If $\omega$ (resp. $\zeta$) is a form in  $E^{p,p}(X_\matR) $ (resp. a current in $D^{p,p}(X_\matR)$), we write $\tilde \omega$ (resp. $\tilde \zeta$) for its class
 in $\tilde E^{p,p}(X_\matR) $ (resp. $\tilde D^{p,p}(X_\matR) $).
 
We briefly recall the definition of the arithmetic Chow groups of $X$, as given in \cite{gilsou}, section 3.3. Let
$Z^q(X)$ denote the group of cycles of codimension $q$ in $X$ and $\textnormal{CH}^q (X) $ denote the $q$-th Chow group of $X$.
 We write $ \widehat{Z}^q(X) $ for the subgroup of
\[
  Z^q(X)\oplus\tilde D^{q-1,q-1}(X_\matR)
\]
consisting of pairs $(z,\tilde h)$ where $z\in Z^q(X) $ and 
$h\in D^{q-1,q-1}(X_\matR)$ verify
\[ \textnormal{dd}^c h+\delta_z \in E^{q,q}(X_\matR).
\]
By definition, the class $\tilde h$ is then a Green current for $z$. Note that if $\tilde h$ is a Green current for $Z$, the form $\textnormal{dd}^c h+\delta_Z  $ is closed.
                                                                                                           


For any codimension $q-1$ integral subscheme $i: W\hookrightarrow X$ and any $f\in k(W)^*$,   
one can verify, by means of the Poincaré-Lelong lemma, that the pair
\[
 \widehat{\text{div}}(f):=(\text{div}(f), -i_* \text{log}|f|^2)
\]
is an element in $ \widehat{Z}^q(X) $.
Then the
q-th arithmetic Chow group of $X$ is the quotient
\[\widehat{\textnormal{CH}}^q (X):= \widehat{Z}^q(X)/\widehat{R}^q(X)
 \]
where $\widehat{R}^q(X)$ is the subgroup generated by all pairs $ \widehat{\text{div}}(f)$, for any $f\in k(W)^*$ and 
any $W\subset X$
as above. If $Z^{q,q}(X_\matR)\subseteq E^{q,q}(X_\matR)$ denotes the subspace of closed forms, we have a well defined map
\[
   \omega:  \widehat{\textnormal{CH}}^q (X)\rightarrow Z^{q,q}(X_\matR)                                                                                                                     
     \]        
     sending the class of $(z,\tilde h)$ to $\textnormal{dd}^c h+\delta_z $. Finally we have a map
     \[
      \zeta: \widehat{\textnormal{CH}}^q(X)\rightarrow \textnormal{CH}^q(X)
     \]
sending the class of $(z,\tilde h)$ to the class of $z$.
\subsection{Analytic Deligne cohomology of arithmetic varieties}
If $X$ is an arithmetic variety over $R$ we write
\[
\textnormal{H}^{q}_{D^\textnormal{an}}(X_{\matR}, \matR(p)):=\{\gamma \in \textnormal{H}^{q}_{D^\textnormal{an}}(X(\matC), 
\matR(p))|F^*_\infty \gamma=(-1)^{p}\gamma \},\] where 
$  \textnormal{H}^{*}_{D^\textnormal{an}}(X(\matC), 
\matR(p))$ is the analytic Deligne cohomology of the complex manifold $X(\matC) $, i.e. the hypercohomology of the complex 
 \begin{center}
\begin{tikzpicture}[description/.style={fill=white,inner sep=2pt}, ciao/.style={fill=red,inner sep=2pt}]
\matrix (m) [matrix of math nodes, row sep=3.5em,
column sep=2.5em, text height=1.5ex, text depth=0.25ex]
{ 0  & (2\pi i)^p \matR & \mathcal{O}_{X(\matC)} & \Omega^1_{X(\matC)} & ... & \Omega_{X(\matC)}^{p-1} & 0, \\
  };
	\path[->,font=\scriptsize]
		(m-1-3) edge node[above] {$d$} (m-1-4)
(m-1-1) edge node[above] {$ $} (m-1-2)
(m-1-2) edge node[above] {$ $} (m-1-3)
(m-1-4) edge node[above] {$ $} (m-1-5)
(m-1-5) edge node[above] {$ $} (m-1-6)
(m-1-6) edge node[above] {$ $} (m-1-7)
		;
\end{tikzpicture}
\end{center}
($  \Omega_{X(\matC)}^{*}$ denotes the De Rham complex of holomorphic forms on $X(\matC)$). In the following sections we will need the following characterization 
(cf. [\cite{burgosarith}, Section 2]):
 \begin{equation}\label{eq:explicit-deligne-coh} \textnormal{H}^{2p-1}_{D^\textnormal{an}}(X_\matR, \matR(p))= \{ \tilde x\in  (2\pi i)^{p-1}\tilde E^{p-1,p-1}(X_\matR)| \ \partial \bar \partial x=0   \}.
\end{equation}
\subsection{Analytic Deligne cohomology and Betti cohomology}
By definition of analytic Deligne cohomology there is a canonical map to Betti cohomology
\[
 \phi_B:  H^{2g-1}_{D^{\textnormal{an}}}( (\mathcal{A}\setminus \mathcal{A}[N])_\matR,\matR(g))\rightarrow 
  H^{2g-1}_{B}( (\mathcal{A}\setminus \mathcal{A}[N])(\matC),\matR(g)).
\]
We will need later the following explicit description of this map: First we compute the group $ H^{2g-1}_{B}( (\mathcal{A}\setminus \mathcal{A}[N])(\matC),g)$ with the cohomology of the  complex of currents $D^*((\mathcal{A}\setminus \mathcal{A}[N])(\matC), g):=
 (2\pi i)^g D^*((\mathcal{A}\setminus \mathcal{A}[N])(\matC))$
 \[
   H^{2g-1}_{B}( (\mathcal{A}\setminus \mathcal{A}[N])(\matC),g)=\frac{\{\eta\in D^{2g-1}((\mathcal{A}\setminus \mathcal{A}[N])(\matC),g)\lvert d \eta=0\} }
   {\{d\omega\lvert \omega\in D^{2g-2}((\mathcal{A}\setminus \mathcal{A}[N])(\matC), g)\}}.
 \]
\begin{lemma}
Using the description in equation \eqref{eq:explicit-deligne-coh}, the map $\phi_B$ sends the class $\tilde x$ of  $x\in (2\pi i)^{g-1}E^{g-1,g-1}((\mathcal{A}\setminus \mathcal{A}[N])_\matR)$ with $\partial \bar \partial x=0$ to 
\begin{equation} 
\phi_B(\tilde{x})=[4\pi i \textnormal{d}^cx]
\end{equation}
\end{lemma}
\begin{proof}
This is \cite[Theorem 2.6]{burgosarith}.
\end{proof}
We also need an explicit description of the connecting homomorphism, which we call the residue homomorphism
\[
 \res_B: H^{2g-1}_{B}( (\mathcal{A}\setminus \mathcal{A}[N])(\matC),\matR(g))\rightarrow H^{2g}_{B, \mathcal{A}[N]\setminus S_0}( (\mathcal{A}\setminus S_0)(\matC),\matR(g)).
\]
For this we compute $ H^{2g}_{B, \mathcal{A}[N]\setminus S_0}( (\mathcal{A}\setminus S_0)(\matC),\matR(g))$ with the cohomology of the simple complex of the restriction morphism
$ D^*((\mathcal{A}\setminus S_0)(\matC), g)\rightarrow D^*((\mathcal{A}\setminus \mathcal{A}[N])(\matC), g) $ and get
\small{
\begin{multline*} H^{2g}_{B, \mathcal{A}[N]\setminus S_0}( (\mathcal{A}\setminus S_0)(\matC),\matR(g))=\\
 \frac{\{(\xi,\tau)\in D^{2g}((\mathcal{A}\setminus S_0)(\matC),g) )
  \oplus  D^{2g-1}((\mathcal{A}\setminus \mathcal{A}[N])(\matC),g)\lvert d\xi=0 \text{ and } \xi\lvert_{\mathcal{A}\setminus \mathcal{A}[N]}=d\tau\}}
  {\{(d\theta,\theta\lvert_{\mathcal{A}\setminus \mathcal{A}[N]}-d\alpha)\lvert \theta\in D^{2g-1}((\mathcal{A}\setminus S_0)(\matC),g) , \alpha \in
  D^{2g-2}((\mathcal{A}\setminus \mathcal{A}[N])(\matC),g)\}}.
\end{multline*}
}\normalsize{ Note that we are using the simple complex  as in [\cite{burgosarith}, Section 1] (and not the cone in the sense of Verdier) of the restriction morphism to compute cohomology with support.} From the definitions one gets immediately:
\begin{lemma}\label{lemma:res-B}
The residue $\res_B$ sends the class of $\eta$, which we denote by $[\eta]$, to 
\[ 
\res_B([\eta])=[0,-\eta],
\]
where $[0,-\eta]$ denotes the class of $(0,-\eta)$.
\end{lemma}
\section{Review of the polylog and the Maillot-R\"ossler current}
\subsection{The axiomatic definition of $\rm{pol}^0 $}
In \cite{kingsross}, G. Kings and D. R\"ossler have provided a simple axiomatic description of the degree zero part of the polylogarithm on abelian schemes. We briefly recall it here.

 The zero step of the motivic polylogarithm 
 is by definition a class in motivic cohomology
\[
 \textnormal{pol}^0\in H^{2g-1}_{\mathcal{M}}(\mathcal{A}\setminus\mathcal{A}[N],g)^{(0)}.
\]
To describe it more precisely, consider the residue map along $\mathcal{A}[N]$
\[
 H^{2g-1}_{\mathcal{M}}(\mathcal{A}\setminus\mathcal{A}[N],g)\rightarrow H^0_{\mathcal{M}}(\mathcal{A}[N]\setminus S_0,0).
\]
This map induces an isomorphism
\[
 \textnormal{res}:H^{2g-1}_{\mathcal{M}}(\mathcal{A}\setminus\mathcal{A}[N],g)^{(0)}\cong H^0_{\mathcal{M}}(\mathcal{A}[N]\setminus S_0,0)^{(0)}
\]
(see Corollary 2.2.2 in \cite{kingsross}).
\begin{defi}
The degree zero part of the polylogarithm $\rm{pol}^0$ is the unique element of $H^{2g-1}_{\mathcal{M}}(\mathcal{A}\setminus\mathcal{A}[N],g)^{(0)}$ mapping under $\textnormal{res}$ to the fundamental class $1_N^\circ$ of $\mathcal{A}[N]\setminus S_0$.

\end{defi} 
We recall now that we have a map $\textnormal{reg}_{\textnormal{an}}$ defined as the composition
\begin{equation*} 
\begin{tikzpicture}[scale=0.5][description/.style={fill=white,inner sep=0.5pt}, ciao/.style={fill=red,inner sep=2pt}]
\matrix (m) [matrix of math nodes, row sep=3em,
column sep=1.5em, text height=1.5ex, text depth=0.25ex]
{ H^{2g-1}_{\mathcal{M}}(\mathcal{A}\setminus\mathcal{A}[N],g) &  \textnormal{H}^{2g-1}_{D}((\mathcal{A}\setminus\mathcal{A}[N])_{\matR},\matR(g)) &
\textnormal{H}^{2g-1}_{D^\textnormal{an}}((\mathcal{A}\setminus\mathcal{A}[N])_{\matR},\matR(g)) \\
  };
	\path[->,font=\scriptsize] 
		(m-1-1) edge node[above] {$ \textnormal{reg} $} (m-1-2)
		(m-1-2) edge node[above] {$ \textnormal{forget} $} (m-1-3)
		;
\end{tikzpicture} 
\end{equation*} 
where $ \textnormal{reg}$ is the regulator map into Deligne-Beilinson
cohomology and the second map is the forgetful map from Deligne-Beilinson cohomology
to analytic  Deligne cohomology. 
\subsection{The Maillot-R\"ossler current $\mathfrak{g}_{\mathcal{A}^\vee}$}
 In \cite{onacan}, V. Maillot and D. R\"ossler proved the following theorem.
 \begin{theo}[\cite{onacan}, Theorem 1.1]\label{ga}
  There exists a unique a class of currents $\mathfrak{g}_{\mathcal{A}^\vee} \in \tilde D^{g-1,g-1}(\mathcal{A}_\matR)$ 
  which satisfies the following three properties:
  \begin{enumerate}
   \item  $\mathfrak{g}_{\mathcal{A}^\vee}$  is a Green current for $S_0 $,
   \item    $(S_0,\mathfrak{g}_{\mathcal{A}^\vee})=(-1)^g p_{1,\ast}
  \left(\widehat{\textnormal{ch}}(\bar{\mathcal{P}})\right)^{(g)}$ 
 in the group $\widehat{\textnormal{CH}}^g (\mathcal{A})_\matQ$,
   \item $[n]_*\mathfrak{g}_{\mathcal{A}^\vee}=\mathfrak{g}_{\mathcal{A}^\vee}$ for all  $n>0$.
  \end{enumerate}
 \end{theo}
  Here $\bar{\mathcal{P}} $ is the Poincaré bundle on $\mathcal{A}\times_S \mathcal{A}^\vee$ equipped 
 with a canonical hermitian metric,
 $p_1:\mathcal{A}\times_S \mathcal{A}^\vee \rightarrow \mathcal{A}$ 
is the first projection and  $\widehat{\textnormal{CH}}^g (\mathcal{A})$ denotes the $g^\text{th}$ aritmetic Chow group of $ \mathcal{A}$. 
The term $\widehat{\textnormal{ch}}(\bar{\mathcal{P}}) $ is the 
 arithmetic Chern character
of $\bar{\mathcal{P}} $. 

We now consider  the following arithmetic cycle
\[
 (_NS_0,_N\mathfrak{g}_{\mathcal{A}^\vee}):=([N]^* -N^{2g})(S_0, \mathfrak{g}_{\mathcal{A}^\vee}).
\]
Thanks to the geometry of the Poincaré bundle one can show that the class of $ (_NS_0,_N\mathfrak{g}_{\mathcal{A}^\vee})$ 
in $\widehat{\textnormal{CH}}^g (\mathcal{A})_\matQ$ is zero  (cf. Proposition 5.2 in \cite{farak}).
In particular, $\textnormal{dd}^c(_N\mathfrak{g}_{\mathcal{A}^\vee}\rvert_{\mathcal{A}\setminus \mathcal{A}[N]})=0 $ and 
by Theorem 1.2.2(i) in \cite{gilsou} there exists a smooth form in the class of currents 
$_N\mathfrak{g}_{\mathcal{A}^\vee}\rvert_{\mathcal{A}\setminus \mathcal{A}[N]} $. Equivalently, 
$_N\mathfrak{g}_{\mathcal{A}^\vee}\rvert_{\mathcal{A}\setminus \mathcal{A}[N]} $ lies in the image of the inclusion
\[
 \tilde E^{g-1,g-1}((\mathcal{A}\setminus \mathcal{A}[N])_\matR)\hookrightarrow  \tilde D^{g-1,g-1}((\mathcal{A}\setminus \mathcal{A}[N])_\matR).
\]
Since by \eqref{eq:explicit-deligne-coh} the group $H^{2g-1}_{D^{\text{an}}}( (\mathcal{A}\setminus \mathcal{A}[N])_{\matR},\matR(g))$ 
can be represented by classes in $(2\pi i)^{g-1}\tilde E^{g-1,g-1}((\mathcal{A}\setminus \mathcal{A}[N])_\matR)$  with $\textnormal{dd}^c $ equal to zero, we get that 
\begin{lemma} The Maillot-R\"ossler current defines a class
\[ 
(2\pi i)^{g-1} (_N\mathfrak{g}_{\mathcal{A}^\vee})\rvert_{\mathcal{A}\setminus \mathcal{A}[N]}\in H^{2g-1}_{D^{\text{an}}}( (\mathcal{A}\setminus \mathcal{A}[N])_{\matR},\matR(g)).
\]
\end{lemma}

The exact sequence (cf. Theorem 3.3.5 and Remark in \cite{gilsou}) 
\begin{center}
\begin{tikzpicture}[description/.style={fill=white,inner sep=2pt}, ciao/.style={fill=red,inner sep=2pt}]
\matrix (m) [matrix of math nodes, row sep=3.5em,
column sep=2em, text height=1.5ex, text depth=0.25ex]
{  
  H^{2g-1}_{\mathcal{M}}(\mathcal{A}\setminus \mathcal{A}[N] ,g) & H^{2g-1}_{D^{\text{an}}}( (\mathcal{A}\setminus \mathcal{A}[N])_{\matR},\matR(g)) & 
 \widehat{\textnormal{CH}}^g ( \mathcal{A}\setminus  \mathcal{A}[N])_\matQ   
 \\   
  };
	\path[->,font=\scriptsize]
		(m-1-1) edge node[auto
		] {$ $} (m-1-2)
		(m-1-1) edge node[auto] {$\textnormal{reg}_{\textnormal{an}} $} (m-1-2)
		(m-1-2) edge node[auto] {$r  $} (m-1-3)
		;
\end{tikzpicture}
\end{center}
where $r$ sends $\tilde x$ to the class of $\left(0,\frac{\tilde x}{(2\pi i)^{g-1}} \right)$, together with the vanishing of $(_NS_0,_N\mathfrak{g}_{\mathcal{A}^\vee})$ in $\widehat{\textnormal{CH}}^g ( \mathcal{A}\setminus  \mathcal{A}[N])_\matQ$
then implies that the Maillot-R\"ossler current is motivic, i.e., 
\[
(2\pi i)^{g-1}(_N\mathfrak{g}_{\mathcal{A}^\vee})|_{\mathcal{A}\setminus \mathcal{A}[N]}\in 
\text{reg}_{\text{an}}\left(H^{2g-1}_{\mathcal{M}}( \mathcal{A}\setminus \mathcal{A}[N],g)\right).
\]
Since the operator $\text{tr}_{[a]}$ defined in \eqref{eq:tr-def} obviously operates on analytic Deligne cohomology and the map $\text{reg}_{\text{an}} $ intertwines this operator with 
$\text{tr}_{[a]}$, we deduce  from property (3) in the definition of $\mathfrak{g}_{\mathcal{A}^\vee} $ the fact:
\begin{lemma} The Maillot-R\"ossler current is in the image of the regulator from $H^{2g-1}_{\mathcal{M}}( \mathcal{A}\setminus \mathcal{A}[N],g)^{(0)}$:
\[
 (2\pi i)^{g-1}(_N\mathfrak{g}_{\mathcal{A}^\vee})|_{\mathcal{A}\setminus \mathcal{A}[N]}\in \text{reg}_{\text{an}}\left(H^{2g-1}_{\mathcal{M}}( \mathcal{A}\setminus \mathcal{A}[N],g)^{(0)}\right).
\]
\end{lemma}

\section{The comparison between $\textnormal{pol}^0$ and the class $\mathfrak{g}_{\mathcal{A}^\vee}$}
In this section we give an easy conceptual proof of the comparison result between $\textnormal{pol}^0$ and the class $\mathfrak{g}_{\mathcal{A}^\vee} $.
\subsection{A commutative diagram} 

The following lemma, proved by R\"ossler and Kings, is the key for the proof of our comparison result.
\begin{lemma}[\cite{kingsross}, Lemma 4.2.6] \label{applic}
 The following diagram is commutative
  \small{\begin{center}
\begin{tikzpicture}[description/.style={fill=white,inner sep=2pt}, ciao/.style={fill=red,inner sep=2pt}]
\matrix (m) [matrix of math nodes, row sep=3.5em,
column sep=2.5em, text height=1.5ex, text depth=0.25ex]
{ H^{2g-1}_{\mathcal{M}}( \mathcal{A}\setminus \mathcal{A}[N],g)^{(0)}  &  H^{2g}_{\mathcal{M}, \mathcal{A}[N]\setminus S_0}( \mathcal{A}\setminus S_0,g)^{(0)} \\
 H^{2g-1}_{D^{\textnormal{an}}}( (\mathcal{A}\setminus \mathcal{A}[N])_\matR,\matR(g)) & 
 \\  H^{2g-1}_{B}( (\mathcal{A}\setminus \mathcal{A}[N])(\matC),\matR(g))
& H^{2g}_{B, \mathcal{A}[N]\setminus S_0}( (\mathcal{A}\setminus S_0)(\matC),\matR(g)) \\
  };
	\path[->,font=\scriptsize]
		(m-1-1) edge node[right] {$ \textnormal{reg}_{\textnormal{an}}$} (m-2-1)
		(m-1-2) edge node[right] {$ \textnormal{reg}_{\textnormal{B}}$} (m-3-2)
		(m-1-1) edge node[above] {$\res $} (m-1-2)
		(m-1-1) edge node[below] {$\simeq $} (m-1-2)
		(m-2-1) edge node[right] {$\phi_B $} (m-3-1)
		;
	\path[->,font=\scriptsize] 
		(m-3-1) edge node[above] {$  \res_B$} (m-3-2)
		;
\end{tikzpicture}
\end{center}}\normalsize{
and the map $ \textnormal{reg}_{\textnormal{B}}$ is injective.}
\end{lemma}
\subsection{The comparison result} We are now ready to reprove the  comparison result of Kings and R\"ossler.
\begin{theo}[\cite{kingsross}]\label{theo}
  We have the equality
   \[
  -2 \cdot \textnormal{reg}_{\textnormal{an}}(\textnormal{pol}^0)=(2\pi i)^{g-1}(_N\mathfrak{g}_{\mathcal{A}^\vee})|_{\mathcal{A}\setminus \mathcal{A}[N]}.
  \]
  \end{theo}
\begin{proof}
 Let $\psi\in H^{2g-1}_{\mathcal{M}}( \mathcal{A}\setminus \mathcal{A}[N],g)^{(0)}$ be such that 
 $ \text{reg}_{\text{an}}(\psi)=-\frac{(2\pi i)^{g-1}}{2}(_N\mathfrak{g}_{\mathcal{A}^\vee})|_{\mathcal{A}\setminus \mathcal{A}[N]} $. 
 By Lemma \ref{applic}, it is sufficient to show that $\psi$ and $\textnormal{pol}^0 $ have the same image under $ \textnormal{reg}_{\textnormal{B}}\circ \textnormal{res} $.
Now, by definition of $\textnormal{pol}^0 $ we have 
\[
  \textnormal{reg}_{\textnormal{B}}(\textnormal{res}(\textnormal{pol}^0))=\textnormal{reg}_{\textnormal{B}}(1_N^\circ)=\left[(2\pi i)^g\delta_{1_N^\circ},0\right],
\]
and by the description of $\res_B$ in Lemma \ref{lemma:res-B}
 we have 
 \begin{align*}
 \textnormal{reg}_{\textnormal{B}}(\textnormal{res}(\psi))=\res_B(\phi_B(\textnormal{reg}_{\textnormal{an}}(\psi)))&=
\res_B\left(\left[-(2\pi i)^g \textnormal{d}^c\left(_N\mathfrak{g}_{\mathcal{A}^\vee}|_{\mathcal{A}\setminus \mathcal{A}[N]}\right)\right] \right)\\
 &=\left[0,(2\pi i)^g
 \textnormal{d}^c(_N\mathfrak{g}_{\mathcal{A}^\vee})|_{\mathcal{A}\setminus \mathcal{A}[N]}\right].
\end{align*}
The difference  $ \textnormal{reg}_{\textnormal{B}}(\text{res}((\textnormal{pol}^0-\psi))) $ is then represented by the pair 
\[\left((2\pi i)^g\delta_{1_N^\circ}, - (2\pi i)^g
 \textnormal{d}^c(_N\mathfrak{g}_{\mathcal{A}^\vee})|_{\mathcal{A}\setminus \mathcal{A}[N]}\right),\]
which is a coboundary, since (by [\cite{farak}, Proposition 5.2])
\[(2\pi i)^{g}\delta_{1_N^\circ}+(2\pi i)^{g}\textnormal{dd}^c(_N\mathfrak{g}_{\mathcal{A}^\vee}|_{\mathcal{A}\setminus S_0})=0.\]
\end{proof}
\section{A decomposition of the arithmetic Chow group}
Recall the exact sequence (cf. Theorem and Remark in 3.3.5 \cite{gilsou}) 
\begin{equation}\label{eq:fundamental-ex-seq}
  H^{2p-1}_{\mathcal{M}}(\mathcal{A},p) \to \widetilde{E}^{p-1,p-1}(\mathcal{A}_{\matR}) \to
 \widehat{\textnormal{CH}}^p ( \mathcal{A})_\matQ  \to 
 {\textnormal{CH}}^p ( \mathcal{A})_\matQ \to 0.
\end{equation}
The endomorphism $[n]^{*}$ acts on this sequence and we want to study the decomposition into generalized eigenspaces. Denote by $E^{p,q}_\mathcal{A}$ the sheaf of $p,q$-forms on $\mathcal{A}(\matC)$.
For the next result observe that we have an isomorphism of sheaves $\pi^{*}\epsilon^{*}E^{p,q}_\mathcal{A}\cong E^{p,q}_{\mathcal{A}}$, which identifies the pull-back of sections of $\epsilon^{*}E^{p,q}_\mathcal{A}$  on the base with the translation invariant differential forms on $\mathcal{A}$. For a $\mathcal{C}^{\infty}$ section $a:S(\matC)\to  \mathcal{A}(\matC)$, we denote by $\tau_a:\mathcal{A}(\matC)\to \mathcal{A}(\matC)$ the translation by $a$. A differential form $\omega$ is translation invariant, if $\tau_a^{*}\omega=\omega$ for all sections $a$.
\begin{theo}\label{eigen}
Let $n\ge 2$ and $\omega \in \widetilde{E}^{p,q}(\mathcal{A}_{\matR})$.
Assume that $\omega$ is a generalized eigenvector for $[n]^{*}$ with eigenvalue $\lambda$, i.e. $([n]^{*}-\lambda)^{k}\omega=0$, for some $k\ge 1$. Then the form $\omega$ is translation invariant. In particular, there is a section $\eta\in \epsilon^{*}{E}^{p,q}_\mathcal{A}(S_\matR)$ with $\omega=\pi^{*}\eta$.
Moreover, one has
$[n]^{*}\omega= n^{p+q}\omega$, i.e., $\omega$ is an eigenvector with eigenvalue $n^{p+q}$.
\end{theo}
\begin{proof} The statement that $\omega$ is translation invariant does not depend on the complex structure. We
use that locally on the base the family of complex tori $\pi:\mathcal{A}(\matC)\to S(\matC)$ is as a $\mathcal{C}^{\infty}$-manifold of the form 
$U\times (S^{1})^{2g}$, where $U\subset S(\matC)$ is open and $S^{1}=\matR/\matZ$ is a real torus. In this situation it suffices to show that $\omega$ is translation invariant under a dense subset of points of 
$(S^{1})^{2g}$. 

We start to prove the following claim: if the form $\eta:=([n]^{*}-\lambda)\omega$ is translation invariant, then $\omega$ is translation invariant.

First note that $\lambda\neq 0$ because $[n]_*[n]^{*}\omega=n^{2g}\omega$, which implies that $[n]^{*}$ is injective.
As the set $\{a\in (S^{1})^{2g}\mid [n^{r}](a)=0\mbox{ for some }r\ge 0\}$
is dense in $(S^{1})^{2g}$, by induction over $r$ it suffices  to show that $\tau_a^{*}\omega=\omega$ for $a$ with $[n^{r}](a)=0$. The case $r=0$ is trivial because then $a=0$. Suppose we know that $\omega$ is translation invariant for all $b$ with $[n^{r-1}](b)=0$ and let $a$ be such that
$[n^{r}](a)=0$. We compute
\[ 
\lambda\tau_a^{*}\omega=\tau_a^{*}([n]^{*}\omega-\eta)=[n]^{*}\tau_{[n]a}^{*}\omega-\tau_a^{*}\eta=[n]^{*}\omega-\eta=\lambda\omega.
\]
As $\lambda\neq 0$, it follows that $\tau_a^{*}\omega=\omega$. This completes the induction step.

We now show by induction on $k$ that $\omega$ with $([n]^{*}-\lambda)^{k}\omega=0$ is translation invariant. For $k=1$ this follows from the claim by setting $\eta=0$. Suppose that all forms $\eta$ with $([n]^{*}-\lambda)^{k-1}\eta=0$ are translation invariant. Then 
$\eta:=([n]^{*}-\lambda)\omega$ is translation invariant and it follows from the claim that also $\omega$ is translation invariant.

For the final statement we just observe that $[n]^{*}$ acts via $n^{p+q}$-multiplication on the bundle $\epsilon^{*}E^{p,q}_\mathcal{A}$ whose sections identify with the translation invariant forms on $\mathcal{A}(\matC)$.
\end{proof}
For the next result we have to consider generalized eigenspaces for 
$[n]^{*}$ and to distinguish these from the generalized eigenspaces for $[n]_*$, we write
\[ 
V(a):=\{v\in V\mid ([n]^*-n^{a})^{k}v=0, \mbox{ for some }k\ge 1\}
\]
\begin{cor}\label{cor:exact-eigenspaces}
For each $a=0,\ldots,2g$ there is an exact sequence
\begin{equation*}
  H^{2p-1}_{\mathcal{M}}(\mathcal{A},p){(a)} \to \widetilde{E}^{p-1,p-1}(\mathcal{A}_{\matR}){(a)} \to
 \widehat{\textnormal{CH}}^p ( \mathcal{A})_\matQ{(a)}  \to 
 {\textnormal{CH}}^p ( \mathcal{A})_\matQ{(a)}
\end{equation*}
of generalized $[n]^{*}$-eigenspaces for the eigenvalue $n^{a}$. In particular, for $a\neq 2(p-1)$ one has an injection
\[ 
\widehat{\textnormal{CH}}^p ( \mathcal{A})_\matQ{(a)}  \hookrightarrow
 {\textnormal{CH}}^p ( \mathcal{A})_\matQ{(a)}. 
\]
\end{cor}
\begin{proof}
The sequence \eqref{eq:fundamental-ex-seq} is a sequence of modules under the principal ideal domain $\matC[X]$, where $X$ acts as $[n]^*$. Note that taking the torsion submodule $TM:=\ker (M\to M\otimes_{\matC[X]}\Quot\matC[X])$ is a left exact functor on short exact sequences 
\[ 
0\to M'\to M\to M''\to 0.
\]
If $M'$ is torsion the functor $T$ is even exact. As $ H^{2p-1}_{\mathcal{M}}(\mathcal{A},p)$ is torsion, the exact sequence \eqref{eq:fundamental-ex-seq} gives rise to an exact sequence
\[ 
0\to T\im( H^{2p-1}_{\mathcal{M}}(\mathcal{A},p))\to T\widetilde{E}^{p-1,p-1}(\mathcal{A}_{\matR})\to T\widehat{\textnormal{CH}}^p ( \mathcal{A})_\matQ \to 
 T{\textnormal{CH}}^p ( \mathcal{A})_\matQ.
\]
As a torsion $\matC[X]$-module is the direct sum of its generalized eigenspaces, the first claim follows. The second statement follows from the first, as $\widetilde{E}^{p-1,p-1}(\mathcal{A}_{\matR}){(a)}=0$ for
$a\neq 2(p-1)$ by Theorem \ref{eigen}.
\end{proof}
\section{An axiomatic  characterization of the Maillot-R\"ossler current }
We want to prove an axiomatic characterization of $(-1)^g p_{1,\ast}
  \left(\widehat{\textnormal{ch}}(\bar{\mathcal{P}})\right)^{(g)}$. The result is the following. 
\begin{theo}\label{uniq}
 Let $\widehat{\xi}$ be an element of $\widehat{\textnormal{CH}}^g (\mathcal{A})_\matQ $ satisfying the following two properties:
\begin{itemize}
 \item $\zeta(\widehat{\xi})=S_0$ in $\textnormal{CH}^g (\mathcal{A})_\matQ $
 \item $([n]^*-n^{2g})^{k}(\widehat{\xi})=0 $ in $\widehat{\textnormal{CH}}^g (\mathcal{A})_\matQ $ for some   $n\geq 2$ and some $k\ge 1$.
\end{itemize}
 Then $\widehat{\xi}=(-1)^g p_{1,\ast}
  \left(\widehat{\textnormal{ch}}(\bar{\mathcal{P}})\right)^{(g)}=(S_0,\mathfrak{g}_{\mathcal{A}^\vee})$.
\end{theo}
\begin{rem}
\begin{enumerate}
 \item 
Notice that, even if $(-1)^g p_{1,\ast}
  \left(\widehat{\textnormal{ch}}(\bar{\mathcal{P}})\right)^{(g)} $ satisfies the second property for every $n$ (cf. Proposition 6 in \cite{farak}), 
  it is sufficient to ask that this property holds for one integer greater than one to uniquely characterize it.
  \item
Notice also that the condition $([n]^*-n^{2g})(\widehat{\xi})=0 $ implies 
$[n]_*(\widehat{\xi})=\widehat{\xi} $, thanks to the projection formula
$[n]_*[n]^*=n^{2g} $. In the case of an abelian scheme over the ring of integers of a number field,  K. K\"unnemann showed that there exists a decomposition of the 
Arakelov Chow groups as a direct sum of eigenspaces for the pullback $[n]^*$ (cf. \cite{kunn}). As a consequence, in this particular case the conditions $([n]^*-n^{2g})(\widehat{\xi})=0 $ and 
$[n]_*(\widehat{\xi})=\widehat{\xi}$ are equivalent, if $\widehat{\xi} $ belongs to the Arakelov Chow group.
\end{enumerate} \end{rem}

\begin{proof}
By definition  $\widehat{\xi}\in \widehat{\textnormal{CH}}^g (\mathcal{A})_\matQ (2g)$ and by Corollary \ref{cor:exact-eigenspaces} one has an injection
\[ 
\widehat{\textnormal{CH}}^g (\mathcal{A})_\matQ(2g) \hookrightarrow {\textnormal{CH}}^g (\mathcal{A})_\matQ (2g).
\]
This shows that $\widehat{\xi}$ is uniquely determined by its image in ${\textnormal{CH}}^g (\mathcal{A})_\matQ $. As this image is the same as that of $(S_0,\mathfrak{g}_{\mathcal{A}^\vee})$, this shows the theorem.
\end{proof}

Theorem \ref{theo} and Theorem \ref{uniq} give us the following axiomatic characterization of $\mathfrak{g}_{\mathcal{A}^\vee} $ and therefore of $\textnormal{pol}^0 $.
 \begin{prop}\label{char}
 The class  $\mathfrak{g}_{\mathcal{A}^\vee}$ is the unique  element $\mathfrak{g}\in  \tilde{D}^{g-1,g-1}(\mathcal{A}_\matR)$ such that
\begin{enumerate}
   \item  $\mathfrak{g}$  is a Green current for $S_0 $, 
   \item     $([n]^*-n^{2g})^k(S_0,\mathfrak{g})=0$ 
 in  $\widehat{\textnormal{CH}}^g (\mathcal{A})_\matQ$ for some $n\geq 2$ and some $k\geq1$,
   \item $[m]_*\mathfrak{g}=\mathfrak{g}$ for some $m>1$.
  \end{enumerate}
 Furthermore, 
$\textnormal{pol}^0$ is the unique element in $H^{2g-1}_{\mathcal{M}}(\mathcal{A}\setminus \mathcal{A}[N] ,g)^{(0)}$ such that
 \[
   -2 \cdot \textnormal{reg}_{\textnormal{an}}(\textnormal{pol}^0)=   (2\pi i)^{g-1}([N]^*\mathfrak{g}_{\mathcal{A}^\vee}-N^{2g}\mathfrak{g}_{\mathcal{A}^\vee})|_{\mathcal{A}\setminus \mathcal{A}[N]}
     \in \textnormal{H}^{2g-1}_{D^\textnormal{an}}((\mathcal{A}\setminus\mathcal{A}[N])_{\matR},\matR(g)).
 \]
  \end{prop} 
\begin{proof}
 By Theorem \ref{uniq} we know that the first two conditions of our proposition are equivalent to the first two conditions in Theorem \ref{ga}, so that  $\mathfrak{g}_{\mathcal{A}^\vee} $ satisfies the three properties above. Suppose now that $\mathfrak{g}\in  \tilde{D}^{g-1,g-1}(\mathcal{A}_\matR)$
 is another element satisfying the  three properties of our propositon and let $m>1$ be such that  $[m]_*\mathfrak{g}=\mathfrak{g}$. We want to show that 
 $\mathfrak{g}_{\mathcal{A}^\vee}=\mathfrak{g}$. Since by Theorem \ref{uniq} $(S_0, \mathfrak{g})=(-1)^g p_{1,\ast}
  \left(\widehat{\textnormal{ch}}(\bar{\mathcal{P}})\right)^{(g)}=(S_0,\mathfrak{g}_{\mathcal{A}^\vee})$, the exact sequence (\ref{eq:fundamental-ex-seq})
  implies that the difference
  $\mathfrak{g}_{\mathcal{A}^\vee}-\mathfrak{g}$ belongs to the image of the regulator $ H^{2g-1}_{\mathcal{M}}(\mathcal{A},g) \to \widetilde{E}^{g-1,g-1}(\mathcal{A}_{\matR})$.
  Since $H^{2g-1}_{\mathcal{M}}(\mathcal{A},g) $ is a torsion module over $\matC[X]$ (with $X$ acting as $[m]^*$), then $\mathfrak{g}_{\mathcal{A}^\vee}-\mathfrak{g}$ lies in
  $T \widetilde{E}^{g-1,g-1}(\mathcal{A}_{\matR})$. The projection formula and Theorem \ref{eigen} give
  \[
   m^{2g}(\mathfrak{g}_{\mathcal{A}^\vee}-\mathfrak{g})=[m]_*[m]^*(\mathfrak{g}_{\mathcal{A}^\vee}-\mathfrak{g})=m^{2g-2}[m]_*(\mathfrak{g}_{\mathcal{A}^\vee}-\mathfrak{g})
  \]
  i.e. $[m]_*(\mathfrak{g}_{\mathcal{A}^\vee}-\mathfrak{g})=m^2(\mathfrak{g}_{\mathcal{A}^\vee}-\mathfrak{g})$, but property three in our proposition implies that
  $[m]_*(\mathfrak{g}_{\mathcal{A}^\vee}-\mathfrak{g})=(\mathfrak{g}_{\mathcal{A}^\vee}-\mathfrak{g}) $. This is possible only if $\mathfrak{g}_{\mathcal{A}^\vee}-\mathfrak{g} $ is zero.

The second statement is a simple consequence of Theorem \ref{theo} and the fact that $ \textnormal{reg}_{\textnormal{an}} $ is injective when restricted to $ H^{2g-1}_{\mathcal{M}}(\mathcal{A}\setminus\mathcal{A}[N],g)^{(0)} $ (cf. Lemma 4.2.6
 in \cite{kingsross}).
\end{proof}

\bibliography{biblioWorkwithKings}{}
\bibliographystyle{alpha}
  \end{document}